\documentclass{amsart}
\usepackage{amsaddr}
\usepackage{graphicx}
\usepackage{amsmath}
\usepackage{epstopdf}
\usepackage{mathptmx}  
\usepackage{amsmath,amssymb,amsthm}
\usepackage{mathrsfs}
\usepackage{xcolor}
\usepackage{textcomp}
\usepackage{manyfoot}
\usepackage{booktabs}
\usepackage{algorithm}
\usepackage{algorithmicx}
\usepackage{algpseudocode}
\usepackage{listings}
\usepackage{stackrel}
\usepackage{tikz} 
\usetikzlibrary{matrix,arrows}
\usepackage{color}
\usepackage{textcomp}
\usepackage{enumerate}
\usepackage{ifthen}
\usepackage{mathtools}
\usepackage{subcaption}
\usepackage{float}
\usepackage{mathrsfs}
\usepackage{booktabs}
\usepackage{xypic}
\usepackage{bbm}
\usepackage[all]{xy}
\usepackage{hyperref}
\usepackage[capitalise]{cleveref}

\newcommand{\norm}[1]{\left\| #1 \right\|}

\DeclareMathAlphabet{\mathcal}{OMS}{cmsy}{m}{n}

\newtheorem{theorem}{Theorem}[section]
\newtheorem{corollary}{Corollary}[theorem]
\newtheorem{lemma}[theorem]{Lemma}

\author{Benjamin Jarman$^1$, Lara Kassab$^{1*}$, Deanna Needell$^1$,\and Alexander Sietsema$^1$}
\address{$^1$University of California, Los Angeles, 520 Portola Plaza, Los Angeles, CA 90025, USA}

\thanks{$^*$Corresponding author. Email: \href{mailto:lkassab@math.ucla.edu}{lkassab@math.ucla.edu}}
\title{Stochastic Iterative Methods for Online Rank Aggregation from Pairwise Comparisons}
\thanks{BJ and DN were partially supported by NSF DMS 2011140 and LK and DN were partially supported by Dunn Family Endowed Chair Fund.}

\begin{document}

\begin{abstract}
In this paper, we consider large-scale ranking problems where one is given a set of (possibly non-redundant) pairwise comparisons and the underlying ranking explained by those comparisons is desired. We show that stochastic gradient descent approaches can be leveraged to offer convergence to a solution that reveals the underlying ranking while requiring low-memory operations. We introduce several variations of this approach that offer a tradeoff in speed and convergence when the pairwise comparisons are noisy (i.e., some comparisons do not respect the underlying ranking). We prove theoretical results for convergence almost surely and study several regimes including those with full observations, partial observations, and noisy observations. Our empirical results give insights into the number of observations required as well as how much noise in those measurements can be tolerated.
\end{abstract}

\maketitle

\section{Introduction}
\label{sec:introduction}
We consider the problem of ranking a collection of $n$ objects using \textit{pairwise comparisons}, that is, information of the form `item $i$ is superior to item $j$'. This problem arises in a wide range of settings: for example, one may wish to rank a league of sports teams, with the available data being the outcomes of two-team matches. Alternatively, a retailer may wish to determine a ranking of their products by surveying customers on their pairwise preferences. Further examples include more general recommender systems \cite{AggRecommender}, determining individuals' perception of urban areas through pairwise street view comparisons \cite{Salesses2013TheCo}, and ranking students in massive online courses via peer grading \cite{Piech2013Tuned}. In all of these settings, the aim is to obtain the ranking using as few comparisons as possible, as there is usually some cost (computational, financial, or otherwise) associated with acquiring or using comparisons.

In each of these applications, the method in which comparisons are sequentially acquired is critical. We focus in this work on the \textit{passive} (or \textit{non-adaptive}) setting, in which the ranker has no control over which pair of items will be compared at any point. This case includes, for example, the setting in which the ranker is given a fixed set of $m$ comparisons, or alternatively an \textit{online} setting in which comparisons are sampled one-by-one and are assumed to be random. This is the case, for example, when ranking sports teams (as the matches to be played are predetermined). This setting is in contrast to the \textit{active} (or \textit{adaptive}) setting, in which the ranker may choose which pairs of objects to compare based on previous comparisons. This setting is common in applications where the ranker has control over data acquisition, for instance in the aforementioned ranking perception of urban areas, and was studied in depth in \cite{HeckelApprox}.

We focus in this work on the aforementioned online setting, and consider settings where the data may be massively large-scale, the comparisons may be non-redundant (each comparison may be observed only one or fewer times), and the comparison outcomes are not necessarily random. We show that a simple and computationally efficient stochastic gradient descent method (SGD), which by nature is highly scalable to the big data regime, can be leveraged to solve the ranking problem. We consider the Kaczmarz method, a particular variant of SGD, as well as other tailored approaches. We give theoretical results showing that our method converges in finite time almost surely, as well as bounding the expected number of iterations to reach convergence. A range of empirical results are also provided. We also consider the case in which comparisons are \textit{noisy}, that is, the reverse of the respective comparison in the true ranking. We also present one adaptation that is robust to this form of noise and provide empirical analyses.

\subsection{Contribution and Organization} 
\label{sec:contribution}
We begin in \cref{sec:background} with the problem formulation, background, and related work of the rank problem and stochastic gradient descent approaches. \cref{sec:howmany} also includes a discussion of how many observations are needed in various settings.  We present the Kaczmarz method approach, which is the classical Kaczmarz method for feasibility, but applied to the ranking problem, and theoretical guarantees in \cref{sec:main}. Our theoretical results show in the setting of full observations that the iterates converge linearly to the feasible region that explains the underlying ranking. We develop a variation of this approach in \cref{sec:noisy} to handle the setting in which some observed comparisons are inconsistent, i.e., they do not respect the underlying ranking. Although this can be considered as ``noise", this leads to multiplicative rather than additive noise. Finally, we showcase empirical results and investigate other step size choices and implementation details in \cref{sec:exps}. We view our work as complementary to previous work that relies on a randomized model for observations and often requires redundant observations (multiple comparisons for each given pair).  Most importantly, we highlight the mathematical behavior of well-known SGD methods for feasible region detection when that feasible region arises from pairwise comparison data, and when such a feasible solution yields an underlying consistent ranking.

\section{Problem Formulation and Background}
\label{sec:background}
Consider a collection of items $[n] = \{1,...,n\}$, where each item has an intrinsic score $x_i \in \mathbb{R}$. We define a \emph{ranking} of these items as a permutation $\pi : [n] \to [n]$ such that $x_{\pi(1)} \geq x_{\pi(2)} \geq ... \geq x_{\pi(n)}$. We then consider the problem of determining this ranking from pairwise comparisons, meaning observations of the form $x_i < x_j$ for some $i,j \in [n]$. We begin by assuming that these pairwise comparisons respect the underlying full ranking (i.e. observations are noiseless), and we seek to recover the full ranking exactly.

In the case of having some fixed number of comparisons $m$, the problem may be formulated as a system of linear inequalities: each pairwise comparison $x_i < x_j$ may be written $x_i - x_j \leq -\varepsilon$ for some $\varepsilon > 0$, and all such comparisons may be compiled into a system $Ax \leq -\varepsilon$, where $A \in \{0,\pm 1\}^{m \times n}$. We introduce $\varepsilon$ as slack to form a system of non-strict inequalities, and it may be chosen to be any positive value: we refer to \cref{sec:exps} for further discussion and implementation considerations. Note that the kernel of $A$ contains $\operatorname{span}\{(1,\dots, 1)\}$, thus the system is underdetermined, and, so long as the underlying graph with items as nodes and edges as comparisons between items is connected, any solution vector will yield the same ranking.  Note that the use of $\varepsilon$ here ensures that a solution to the system will give an unambiguous ranking (see \cref{sec:exps} for further discussion). As an example, upon solving the system

\begin{equation*}
    \begin{bmatrix}
    1 & -1 & 0 & 0 \\
    0 & -1 & 1 & 0 \\
    1 & 0 & -1 & 0 \\
    0 & 0 & -1 & 1 \\
    1 & 0 & 0 & -1 \\
    \end{bmatrix}
    \begin{bmatrix}
    x_1\\
    x_2\\
    x_3\\
    x_4\\
    \end{bmatrix}
    \leq
    -\varepsilon
\end{equation*}
one can deduce the rankings: $x_1 \leq x_4 \leq x_3 \leq x_2$. Note that the values assigned to the solution vector themselves don't carry any particular meaning, we simply find a solution in the feasible region corresponding to all points that would give the desired ranking.

In this work, we consider a general \textit{online} setting, where comparisons are received one-at-a-time and are viewed as being sampled from some distribution $\mathcal{D}$ on the complete set of ${n \choose 2}$ comparisons. We specialise to the case of comparisons being sampled uniformly at random (so that each particular comparison has probability $1/{n \choose 2}$ of being sampled at any particular iteration), but remark that extending our analysis to more general sampling distributions should be straightforward, and that we do not require measurements to be sampled more than once. Letting $Q$ be the matrix formed from every pairwise comparison in the manner described above, this is equivalent to sampling rows from the system $Qx \leq -\varepsilon$. To further the linear algebraic framework, we equivalently refer to comparisons as inequalities of the form $x_i < x_j$, and as vectors $\varphi \in \mathbb{R}^n$ with $i$\textsuperscript{th} entry equal to $1$ and $j$\textsuperscript{th} entry equal to $-1$, with all other entries equal to zero.  This linear system and row sampling duality precisely motivate our use of the Kaczmarz method as a solution, which we provide background for next.

\subsection{Background and Related Work}
\label{sec:relatedwork}
The Kaczmarz method \cite{Kaczmarz1937Angen} (later rediscovered for use in computerized tomography as the Algebraic Reconstruction Technique \cite{Herman1993Algebr}) is a popular iterative method for solving overdetermined consistent linear systems. It was also extended to linear feasibility problems in the classical paper \cite{agmon1954relaxation}.
The Kaczmarz method is a variant of stochastic gradient descent (SGD) with a particular choice of step size. Suppose $A \in \mathbb{R}^{m \times n}, b \in \mathbb{R}^n$ are such that $Ax = b$ is overdetermined with solution $x^\ast$. Then, an arbitrary initial iterate $x^0$ is projected sequentially onto the hyperplanes corresponding to rows of the system $Ax = b$, so that at the $t$\textsuperscript{th} iteration the update has the form 
\[
x^t = x^{t-1} - \frac{a_i^\top x^{t-1} - b_i}{\norm{a_i}^2}a_i,
\]
where $i = t \text{ mod $m$}$. Whilst convergence to $x^\ast$ is guaranteed via a simple application of the Pythagorean theorem, quantitative convergence guarantees proved elusive. In the landmark paper \cite{Strohmer2009Arand}, the authors proved a linear convergence guarantee when rows are selected at random according to a particular distribution. Namely, in their randomized Kaczmarz method, at iteration $t$ row $i$ is selected with probability $\norm{a_i}^2/\norm{A}_F^2$, and the update takes the same form as above. This row selection scheme gave rise to \cref{thm:rkconv}, which shows linear\footnote{Mathematicians sometimes refer to this rate as exponential as opposed to numerical analysts who consider this linear.} convergence to the solution.

\begin{theorem}[\cite{Strohmer2009Arand}]\label{thm:rkconv}
Suppose that $Ax = b$ is consistent with solution $x^\ast$. Then the iterates produced by applying randomized Kaczmarz to this system satisfy:
\[
\mathbb{E}\left(\norm{x^t - x^\ast}^2\right) \leq \left( 1 - \frac{\sigma_{\mathrm{min}}^2}{\norm{A}_F^2}\right)^t \norm{x^0 - x^\ast}^2.
\]
\end{theorem}

This result spurred a boom in related research, including Kaczmarz variants with differing row selection protocols \cite{Haddock2021Greed,Gower2015Random}, block update methods \cite{Needell2014Paved,Necoara2019Faster}, and adaptive methods \cite{Gower2021OnAda}.

\subsection{Additive Noisy Setting}
\label{sec:additivenoise}
Several results have shown convergence of the Kaczmarz method or SGD more generally in the case when additive noise is added to the right-hand side of $Ax = b$. For the Kaczmarz update, the iterates converge linearly to the least squares solution up to some radius that depends on the norm of the noise \cite{needell2010randomized,schopfer2022extended}. There is a body of work on handling even arbitrarily large levels of additive noise as well \cite{zouzias2013randomized,quantHNRS20}. However, the noise we consider in the ranking setting is not additive, as it involves a ``flip" of the inequality, or equivalently multiplication of the rows of the matrix $A$ by $\pm 1$. Existing work in the setting of multiplicative noise typically focuses on motivations from deep learning \cite{wu2019multiplicative,hodgkinson2021multiplicative}, whereas here we consider a very specific noise model arising from errors in the
comparisons.

\subsection{Kaczmarz for Feasibility}
\label{sec:kaczmarzfeasibility}
Stochastic gradient descent, and in particular the Kaczmarz method, have also seen broad use for systems of linear inequalities and other feasibility problems \cite{Leventhal2010RandomizedMF,DeLoeraMotzkin16}. In its simplest form, the Kaczmarz method for inequalities acts essentially the same way as in the setting of equality constraints, except that no projection is made if the constraint is already satisfied. Otherwise, a projection is made onto the space defining that constraint. As in the case of linear equalities \cite{cai2012exponential}, stochastic gradient descent step sizes may be chosen to perform an over-projection, taking the iterate farther into the feasible region, or even under-projection, stopping short of the feasible region. Leventhal and Lewis proved that the Kaczmarz method for inequalities (see \cref{alg:kaczrank}), introduced by Agmon \cite{agmon1954relaxation}, offers convergence with the same rate as in the setting of equalities \cite{Leventhal2010RandomizedMF}.

\subsection{Rank Aggregation}
\label{sec:rankaggregation}
Rank aggregation from pairwise comparisons or preferences has
a wide range of applications in recommendation systems, competitions, information retrieval, and elsewhere. The literature on the topic is vast, and thus in this section, we discuss only those works most related to ours.

Support Vector Machines (SVMs), a popular class of supervised machine learning algorithms for classification and regression tasks, have been successfully applied in learning retrieval functions \cite{herbrich1999support,herbrich2000large,joachims2002optimizing}.
SVMs can be used to learn a linear scoring function for pairwise comparisons; for example, for automatically optimizing the retrieval quality of search engines using clickthrough data \cite{joachims2002optimizing}.
In \cite{wauthier2013efficient}, two simple algorithms for efficient ranking from pairwise comparisons based on scoring functions are proposed. One predicts rankings with approximately uniform quality across the ranking, while the other predicts the true ranking with higher quality near the top of the ranking than the bottom. It is shown that the algorithms in expectation achieve a lower bound on the sample complexity for predicting a ranking with fixed expected Kendall tau
distance. As such, they are competitive alternatives to the SVM, which also achieves the lower bound.

In \cite{negahban2012iterative}, an iterative aggregation algorithm for extracting scores of
objects given noisy pairwise comparisons is proposed where the algorithm has a natural random walk interpretation over
the graph of objects.
The efficacy of the algorithm is studied by analyzing its performance when data is generated under the Bradley-Terry-Luce (BTL) model.
The robustness of rank aggregation from pairwise comparisons in the presence of adversarial corruptions is initiated and studied in \cite{agarwal2020rank} under
the BTL model.
A strong contamination model is studied, where an adversary having complete knowledge of the initial truthful data and the true
BTL weights, can corrupt this data.

In \cite{Borkar2013Rand}, the authors also employ the randomized Kaczmarz method for a certain rank aggregation problem. In particular, they work under the BTL model and use randomized Kaczmarz to solve a linear system originating from the adjacency matrix of a graph arising under this model. They show that the object weights are recovered to arbitrary accuracy. This model and methodology are distinct to ours, and we show that recovering object weights is not sufficient to recover the ranking. We furthermore evaluate our results under different metrics, and focus on the mathematical question of when stochastic gradient-type methods can be used directly on the observations to unveil the underlying ranking via a feasible point. Although ranking problems are our motivation, we believe this mathematical question is interesting in its own right and has a wide array of other applications in feasible region problems.

Note finally that in \cite{radinsky2011ranking}, the authors employ results from statistical learning theory to show that in order to obtain a ranking of $n$ items in which each element is an average of $O(n/C)$ positions away from its position in the optimal ranking, one needs to sample $O(nC^2$) pairs uniformly at random, for any $C > 0$.

\subsection{Necessary and Sufficient Pairwise Comparisons}
\label{sec:howmany}

We divert our attention briefly to the question of how many pairwise comparisons are necessary and sufficient in order to be able to recover the true underlying ranking. 

Mathematically, it must be ensured that the polytope $Qx \leq -\varepsilon$ has dimension one.
To recover the true ranking, it is necessary and sufficient to know the neighbor comparisons $x_{\pi(j+1)} > x_{\pi(j)}$ for $j = 1, ..., n-1$ that we refer to as `backbone' of the ranking.
Given this, we can quantify how many comparisons we need to sample, depending on how the sampling takes place. 

We consider some simple examples below:

\begin{enumerate}[(i)]
    \item If a knowledgeable friend is providing the comparisons, we may obtain the backbone in $n-1$ samples.
    \item If a knowledgeable adversary is providing the comparisons, they may withhold a backbone comparison until the last sample, requiring the full ${n \choose 2}$ comparisons.
    \item If comparisons are sampled uniformly with replacement from the full set of ${n \choose 2}$ possible comparisons, then this is a variation on the coupon collector problem: we need to collect a specified subset of $n-1$ coupons from a set of ${n \choose 2}$ total. The expected number of samples needed to do so is
    \[
    \frac{n(n-1)}{2}\sum_{i=1}^{n-1} \frac{1}{i} = \mathcal{O}(n^2 \log n).
    \]
    \item If comparisons are sampled uniformly \textit{without} replacement, this is a less-studied variation on the coupon collector problem. To compute the expected number of samples needed, consider the following restatement of the problem: in a random binary string consisting of $n$ ones and $m$ zeros, what is the expected position of the last zero? 
    
    We may solve this by considering the average number of ones between two zeros, i.e., the average length of a `run' of ones. After placing $m$ zeros, there are $m+1$ slots to place ones, and $n$ of them to place (we do not require that no two zeroes be adjacent). Thus the average number of ones between two zeros is $n/(m+1)$. Hence, the expected position of the last zero is $n+m - n/(m+1)$.
    
    Porting this back to our setting, we expect to need to sample $(n-1)(n/2 - 1/2 + 1/n) = \mathcal{O}(n^2)$ comparisons.
    \item If we are able to choose which comparisons we want to obtain (the previously mentioned \textit{active} setting), then this is equivalent to doing comparison-based sorting, which can be done with $\mathcal{O}(n\log n)$ comparisons (using, for example, merge sort).
\end{enumerate}

The sampling protocol used in practice is dependent on the situation at hand: for example, a wine subscription company could effectively perform merge sort by sending subscribers specific, non-random pairs of wines. However, a college football season with a predetermined schedule is akin to sampling without replacement, as described above.

\section{KaczRank: Method and Theoretical Guarantees}
\label{sec:main}
We introduce KaczRank, a modification of the randomized Kaczmarz for inequalities method introduced in \cite{Leventhal2010RandomizedMF} applied to the system of inequalities induced by the observed pairwise comparisons. We show the method in full detail in \cref{alg:kaczrank}. In the next two sections, we focus on the Kaczmarz version of SGD, and leave a discussion of other step size choices for \cref{sec:stepsizes}. 
\begin{algorithm}
\caption{KaczRank \cite{Leventhal2010RandomizedMF}}\label{alg:kaczrank}
\begin{algorithmic}[1]
\Procedure{KaczRank($\beta$) }{Input: initial iterate $x^0$, comparisons $\{(\varphi^t, -\varepsilon)\}_{t=1}^{T}$}
\For{$t = 1, 2, \dots, T$}
\State{Compute $r^t = \langle \varphi^t, x^{t-1} \rangle + \varepsilon$}
\If{ $r^t > 0$}
\State{Update $x^{t} = x^{t-1} - \frac{\langle \varphi^t, x^{t-1} \rangle + \varepsilon}{\norm{\varphi^t}^2}\varphi^t$}
\Else
\State{$x^t = x^{t-1}$}
\EndIf
\EndFor{}

\Return{$\operatorname{ranking}(x^T)$}
\EndProcedure{}
\end{algorithmic}
\end{algorithm}

We now proceed with a theoretical analysis of applying KaczRank in the setting where the sequence of comparisons $(\varphi^t)_{t=1}^{\infty}$ is formed by drawing comparisons uniformly at random from the full set of $m = {n \choose 2}$ pairwise comparisons of $n$ objects, which as mentioned in \cref{sec:background} is equivalent to sampling rows from the system $Qx \leq - \varepsilon$. 
We begin with our main result, which shows that our iterates get increasingly close in expectation to the feasible region: 

\begin{corollary}\label{cor:l2conv_perfectinfo}
Let $S$ be the feasible region for the system of linear equalities $Qx \leq -\varepsilon$. Then the iterates $(x^t)_{t=1}^{T}$ formed by applying KaczRank, with initial iterate $x^0$, to a sequence of comparisons $(\varphi^t)_{t=1}^T$ sampled uniformly at random from the set of all pairwise comparisons of $n$ objects satisfy
\begin{equation}
    \mathbb{E}[d(x^{t}, S)^2] \leq \left(1 - \frac{n}{2m}\right)^{t} d(x^0, S)^2,
\end{equation}
where $d(x, S) = \inf\{\norm{x-s}_2 : s \in S\}$.
\end{corollary}

We will discuss the implications of this theorem further below. To prove this result, we begin with the following result of Lewis and Leventhal \cite{Leventhal2010RandomizedMF}:

\begin{theorem}[\cite{Leventhal2010RandomizedMF}]\label{thm:initialconvergence}
Let $S$ be the feasible region for the system of linear equalities $Qx \leq -\varepsilon$. Then the iterates $(x^t)_{t=1}^{T}$ formed by applying KaczRank, with initial iterate $x^0$, to a sequence of comparisons $(\varphi^t)_{t=1}^T$ sampled uniformly at random from the set of all pairwise comparisons of $n$ objects satisfy
\begin{equation}
    \mathbb{E}[d(x^t, S)^2] \leq \left(1-\frac{1}{2L^2 m}\right)^{t} d(x^0, S)^2,
\end{equation}
where $L$ is the Hoffman constant for the system $Qx \leq -\varepsilon$, and $d(x, S) = \inf\{\norm{x-s}_2 : s \in S\}$.
\end{theorem}

We are able to estimate the Hoffman constant $L$ for the system $Qx \leq -\varepsilon$. In \cite{PenaHoffman} the authors introduce the following characterization of the Hoffman constant:

\begin{theorem}[\cite{PenaHoffman}]
Suppose $A \in \mathbb{R}^{m \times n}$. Then the Hoffman constant of $A$, $L(A)$, is given by
\begin{equation}
    L(A) = \max_{J \in \mathcal{S}(A)} \frac{1}{\min_{v \in \mathbb{R}^J_{+}, \norm{v} = 1} \norm{A_J^\top v}},
\end{equation}
where $\mathcal{S}(A)$ is the collection of row subsets $J \subseteq \{1, ..., m\}$ such that $A_J(\mathbb{R}^n) + \mathbb{R}^J_{+} = \mathbb{R}^J$, that is, any vector $v \in \mathbb{R}^J$ can be written as the sum of a vector in the image of $A_J$ and a vector in $\mathbb{R}^J_{+}$, the set of vectors in $\mathbb{R}$ with non-negative entries.
\end{theorem}

Note that $\mathcal{S}(A)$ may also be characterized as the collection of row subsets $J \subseteq \{1, ..., m\}$ such that $A_J x < 0$ is feasible. In the case of the system $Qx \leq -\varepsilon$, that is all row subsets.

We are then able to exploit the following additional result of \cite{PenaHoffman}:

\begin{lemma}[\cite{PenaHoffman}]\label{lem:specialhoff}
Suppose that $A \in \mathbb{R}^{m \times n}$ and that $A(\mathbb{R}^n) + \mathbb{R}^m_{+} = \mathbb{R}^m$. Then
\begin{equation}
    L(A) = \frac{1}{\min_{v \in \mathbb{R}^m_{+}, \norm{v} = 1} \norm{A^\top v}}.
\end{equation}
\end{lemma}

Our matrix $Q \in \mathbb{R}^{{n \choose 2} \times n}$ satisfies the requirements of \cref{lem:specialhoff}. Note furthermore that we can lower bound the above numerator by
\begin{equation}
    \min_{v \in \mathbb{R}^m_{+}, \norm{v} = 1} \norm{Q^\top v} \geq \min_{v \in \mathbb{R}^m, \norm{v} = 1} \norm{Q^\top v} = \sigma_{\mathrm{min}}(Q^\top) = \sigma_{\mathrm{min}}^{+}(Q),
\end{equation}
where the last quantity denotes the smallest positive singular value of $Q$. 
This singular value can be
computed directly, and the original optimization problem is also solvable. One interesting method
is to note that $Q$ is the incidence matrix of the graph $\mathcal{K}_n$, the complete graph on $n$ vertices. We have that $Q\top Q = L_n$, where $L_n$ is the unweighted graph Laplacian matrix \cite{godsil2001algebraic2} of $\mathcal{K}_n$. Therefore, $\sigma_{\mathrm{min}}^{+}(Q) = \sqrt{\lambda_{\mathrm{min}}^{+}(L_n)}$, where $\lambda_{\mathrm{min}}^{+}(L_n)$ denotes the smallest positive eigenvalue of $L_n$, also known as the \emph{algebraic connectivity} of $\mathcal{K}_n$. It is a standard result of spectral graph theory that the algebraic connectivity of $\mathcal{K}_n$ is equal to $n$. Therefore,
\begin{equation}
    L(Q) \leq \frac{1}{\sqrt{n}},
\end{equation}
and thus \cref{cor:l2conv_perfectinfo} follows immediately from \cref{thm:initialconvergence}.

Whilst \cref{cor:l2conv_perfectinfo} shows that our iterates approach the feasible region in expectation in the $2-norm$, there is no direct relationship between the $2-norm$ and the Hamming distance between the rankings of our iterates and the true ranking, that is, the number of items ranked incorrectly. For instance, if the true ranking is $x_1 < x_2 < ... < x_n$, there are vectors arbitrarily close in norm to this cone with the exact reverse ranking. Thus, a more careful consideration of the geometry is necessary to demonstrate that the rankings implied by our iterates do in fact converge to the true ranking. We provide results showing that convergence is achieved in finitely many iterations almost surely, and give an upper bound on the expected number of iterations needed.

We begin with the following lemma, which demonstrates that if we are able to choose which projections are made (rather than them being random), we can always reach the feasible region in at most $N := {n \choose 2}$ steps.

\begin{lemma}\label{lem:n_step_guarantee}
For any initial iterate $x^0 \in \mathbb{R}^n$, there exists a sequence of $N$ projections $P_{i_1}, ..., P_{i_N}$ such that $P_{i_N}P_{i_{N-1}}\hdots P_{i_1}x^0 \in S$.
\end{lemma}

\begin{proof}
Given any initial iterate $x^0$, projecting onto the $n-1$ equations formed by the comparisons $x_1 < x_n, x_2 < x_n, ..., x_{n-1} < x_n$ will ensure that the $n$\textsuperscript{th} coordinate of $x^{n-1}$ is the largest. One may then project onto $x_1 < x_{n-1}, ..., x_{n-2} < x_{n-1}$ in sequence to ensure the $(n-1)$\textsuperscript{th} coordinate of the resulting iterate is the second largest. Continuing in this fashion ensures that the full ranking is recovered (i.e., the iterate is in $S$) after
\[
\sum_{i=1}^{n-1} i = {n \choose 2} =: N
\]
projections.
\end{proof}

We follow this with a second lemma, which states that we may recover exponential-type bounds on the tail probabilities for the time taken for the iterates produced by KaczRank to reach the feasible region. This is a modification of (\cite{Bullo2020Conv}, Lemma 5).

\begin{lemma}\label{lem:control_tail_probs}
Let $(x^t)_{t=0}^{\infty}$ be the iterates produced by applying KaczRank to the system $Qx \leq -\varepsilon$ formed from all pairwise comparisons, with initial iterate $x^0$. Let $\tau = \inf_{t \geq 0}\{t : x^t \in S\}$. Then for any $k \geq 0$, 
\[
\mathbb{P}(\tau \geq k) \leq \left(1 - \frac{2^N}{n^N(n-1)^N}\right)^{\left\lfloor\frac{k}{N+1}\right\rfloor}.
\]
\end{lemma}

\begin{proof}
By \cref{lem:n_step_guarantee}, for any $l \geq 0$, there exists a sequence of projections $P_{i_l}, P_{i_{l+1}}, ..., P_{i_{l+N-1}}$ such that $P_{i_{l+n-1}}...P_{i_l}x^l \in S$. Thus
\begin{align*}
    \mathbb{P}(\{S \text{ is reached in } [l, l+N]\} | x^l) & \geq \mathbb{P}\left(\cap_{s=l}^{l+N-1}\{\text{row $i_s$ is selected at iteration $s$}\}|x^l\right) \\
    &= \prod_{s=l}^{l+N-1}\mathbb{P}(\{\text{row $i_s$ is selected at iteration $s$}\}|x^l) \\
    &= \frac{2^N}{n^N(n-1)^N}.
\end{align*}
Now let $E_l$ be the event that $S$ is reached in $[l, l+N]$. Then for any $M > 0$, we have

\begin{align*}
    \mathbb{P}(\tau \geq (N+1)M) &= \mathbb{P}\left(\cap_{m=0}^{M-1} E_{m(N+1)}^c\right) \\
    &= \mathbb{P}(E_0^c)\prod_{m=1}^{M-1}\mathbb{P}\left(E_{m(N+1)}^c | \cap_{0 \leq m' < m} E_{m'(N+1)}^c \right) \\
    &\leq \left(1 - \frac{2^N}{n^N(n-1)^N}\right)^M.
\end{align*}
Thus,
\[
\mathbb{P}(\tau \geq k) \leq \mathbb{P}\left(\tau \geq \left\lfloor \frac{k}{N+1}\right\rfloor (N+1)\right) \leq \left(1 - \frac{2^N}{n^N (n-1)^{N}}\right)^{\left\lfloor \frac{k}{N+1}\right\rfloor}.
\]
\end{proof}

We are then able to state our main theorem for this section, which states that the iterates of KaczRank reach the feasible region in finite time almost surely, and gives an upper bound on the expected number of iterations required.

\begin{theorem}\label{thm:main_no_noise}
Let $(x^t)_{t=0}^{\infty}$ be the iterates produced by applying KaczRank to the system $Qx \leq -\varepsilon$ formed from all pairwise comparisons, with initial iterate $x^0$. Let $\tau = \inf_{t \geq 0}\{t : x^t \in S\}$. Then
\begin{enumerate}
    \item $\mathbb{P}(\tau < \infty) = 1$, and
    \item $\mathbb{E}(\tau) \leq \frac{(N+1)n^N(n-1)^N}{2^N}$.
\end{enumerate}
\end{theorem}

\begin{proof}
Part 1 follows immediately from \cref{lem:control_tail_probs}. For part 2, we again use \cref{lem:n_step_guarantee} to obtain
\begin{align*}
    \mathbb{E}(\tau) &= \sum_{k=1}^{\infty} \mathbb{P}(\tau \geq k) \\
    &\leq \sum_{k=1}^{\infty} \left(1 - \frac{2^N}{n^N(n-1)^N}\right)^{\left\lfloor \frac{k}{N+1}\right\rfloor}\\
    &= (N+1)\sum_{k=0}^{\infty}\left(1 - \frac{2^N}{n^N(n-1)^N}\right)^k \\
    &= \frac{(N+1)n^N(n-1)^N}{2^N}.
\end{align*}
\end{proof}

\section{Inconsistent Data}
\label{sec:noisy}

In this section, we consider the scenario in which comparison data may contain \text{noise}, in the sense of some sampled comparisons being the reverse of the corresponding comparison in the underlying ranking. Precisely, we assume that for each time $t = 1,2,...$ we sample $-\varphi^t$ rather than $\varphi^t$ with probability $p \in [0,1/2)$.

It is not difficult to show, using a similar argument to the previous section, that the sequence of iterates formed by applying KaczRank to noisy comparisons will still reach the feasible region at some point. We prove this in \cref{lem:noisy_hits_feasible}.

\begin{lemma}\label{lem:noisy_hits_feasible}
Let $p \in [0,1/2)$, and let $(\varphi^t)_{t=1}^{\infty}$ be a sequence of sampled pairwise comparisons. For each $t$, define $\psi^t$ to be equal to $-\varphi^t$ with probability $p$, and equal to $\varphi^t$ with probability $1-p$. Let $(x^t)_{t=1}^{\infty}$ be the sequence of iterates produced by applying KaczRank to the sequence of observations $(\psi^t)_{t=1}^\infty$ with initial iterate $x^0$, and let $\tau^{\mathrm{noise}} = \inf_{t > 0}\{t: x^t \in S\}$. Then we have
\begin{enumerate}
    \item $\mathbb{P}(\tau^{\mathrm{noise}} < \infty) = 1$, and
    \item $\mathbb{E}(\tau^\mathrm{noise}) \leq \frac{(N+1)n^N(n-1)^N}{2^N(1-p)^N}$
\end{enumerate}
\end{lemma}

\begin{proof}
Note that in the noisy case, \cref{lem:n_step_guarantee} still holds, i.e. it is always possible to reach the feasible region with $N = {n \choose 2}$ projections. Thus, a similar result to \cref{lem:control_tail_probs} holds, with the difference coming in the probability we select the $N$ necessary comparisons to reach $S$. In particular, the probability that we select these $N$ (non-noisy) comparisons is
\[
\frac{(1-p)^N 2^N}{n^N (n-1)^N},
\]
and following similar logic to \cref{lem:control_tail_probs} we have that for any $k \geq 0$,
\[
\mathbb{P}(\tau^{\mathrm{noise}} \geq k) \leq \left(1 - \frac{(1-p)^N 2^N}{n^N (n-1)^N}\right)^{\left\lfloor \frac{k}{N+1}\right\rfloor}.
\]
Now, (1) is immediate, and for (2), we have
\begin{align*}
    \mathbb{E}(\tau^{\mathrm{noise}}) &= \sum_{k=1}^{\infty} \mathbb{P}(\tau^{\mathrm{noise}} \geq k) \\
    &\leq \sum_{k=1}^{\infty} \left(1 - \frac{(1-p)^N 2^N}{n^N (n-1)^N}\right)^{\left\lfloor \frac{k}{N+1}\right\rfloor} \\
    &= \frac{(N+1)n^N(n-1)^N}{2^N(1-p)^N}.
\end{align*}
\end{proof}

Whilst the sequence of iterates is guaranteed to hit the feasible region, it is not guaranteed to stay there. If, say, $x^T \in S$ but $\psi^{T+1} = -\varphi^{T+1}$, then $x^{T+1} \notin S$. To minimise the effect of this, we introduce a variant of \cref{alg:kaczrank} designed to minimise how far away from $S$ (in terms of Hamming distance between rankings) our iterates may become. This method, which we call CautiousRank, proceeds similarly to KaczRank, but will project onto a sample comparison only if both the residual is positive and the Hamming distance between the current iterate and its projection is smaller than some cautiousness parameter $\alpha$. This slows convergence (since when iterates are far from $S$, it could be beneficial to project onto non-noisy comparisons that would induce a large such Hamming distance), but ensures that when iterates are near $S$, they will not wander too far away. We give the method in full in \cref{alg:cautiousrank}.

\begin{algorithm}
\caption{CautiousRank}\label{alg:cautiousrank}
\begin{algorithmic}[1]
\Procedure{CautiousRank($\alpha$) }{Input: initial iterate $x^0$, comparisons $\{(\varphi^t, -\varepsilon)\}_{t=1}^{T}$}
\For{$t = 1, 2, \dots, T$}
\State{Compute $r^t = \langle \varphi^t, x^{t-1} \rangle + \varepsilon$}
\State{Compute $y^t = x^{t-1} - \frac{\langle \varphi^t, x^{t-1} \rangle + \varepsilon}{\norm{\varphi^t}^2}\varphi^t$}
\If{ $r^t > 0$ and $d_H(\operatorname{ranking}(x^t), \operatorname{ranking}(y^t)) < \alpha$}
\State{Update $x^{t} = y^t$}
\Else
\State{$x^t = x^{t-1}$}
\EndIf
\EndFor{}

\Return{$\operatorname{ranking}(x^T)$}
\EndProcedure{}
\end{algorithmic}
\end{algorithm}

\section{Implementation and Experiments}
\label{sec:exps}

\subsection{Relaxation and The Role of $\varepsilon$}
\label{sec:stepsizes}

As discussed in \cref{sec:background}, assembling a collection of pairwise comparisons into a feasibility problem gives rise to a strict system of the form $Qx < 0$. Applying KaczRank to the system directly will (eventually) yield iterates satisfying $Qx = 0$, as projections are made onto the hyperplanes defined by taking the constraints as equalities. To obtain iterates that satisfy the strict inequalities, we introduce some slack in the form of $-\varepsilon$ on the right-hand side, and we note that taking any $\varepsilon > 0$ will suffice: this is the case because our methods seek to recover ranks, rather than any underlying score vector.

An alternative methodology would be to add a relaxation parameter into the update for KaczRank, so that if an unsatisfied constraint is selected, the method updates $x^{t-1}$ to
\[
x^{t} = x^{t-1} - \omega \frac{\langle \varphi^t, x^{t-1}\rangle}{\norm{\varphi^t}^2}\varphi^t,
\]
for some relaxation parameter $\omega \in (1,2)$. This will ensure that the selected constraint holds strictly, and the convergence analysis may be performed in a similar way (in particular, it is known that randomized Kaczmarz converges for relaxation parameters $\omega \in (0,2)$, see e.g. \cite{Necoara2019Faster}). In either case, the objective is to project slightly beyond the chosen hyperplane into the feasible halfspace for that observation. As there is no theoretical difference between the two methodologies, we choose the slack variable approach for intuitive clarity.

\subsection{Experimental Results}
\label{sec:experimentalresults}
In this section, we display a series of experimental results using KaczRank and CautiousRank to solve for an underlying ranking. We measure the accuracy of the estimated ranking using two approaches. The first is simply the Hamming distance between the true ranking and the estimated ranking, defined as the number of locations in which the two rankings are different. Note that even nearly perfect rankings can be far from the true ranking under the Hamming distance (e.g., if the estimated ranking is a shift of the underlying ranking). For this reason, we also utilize the $k$-distance, which is equal to the number of items that are not within $k$ places of their true location. This is a natural generalization of the Hamming distance that, depending on the choice of $k$, tolerates some muddling of the true ranking.

\begin{figure}
    \centering
    \includegraphics[width=3in]{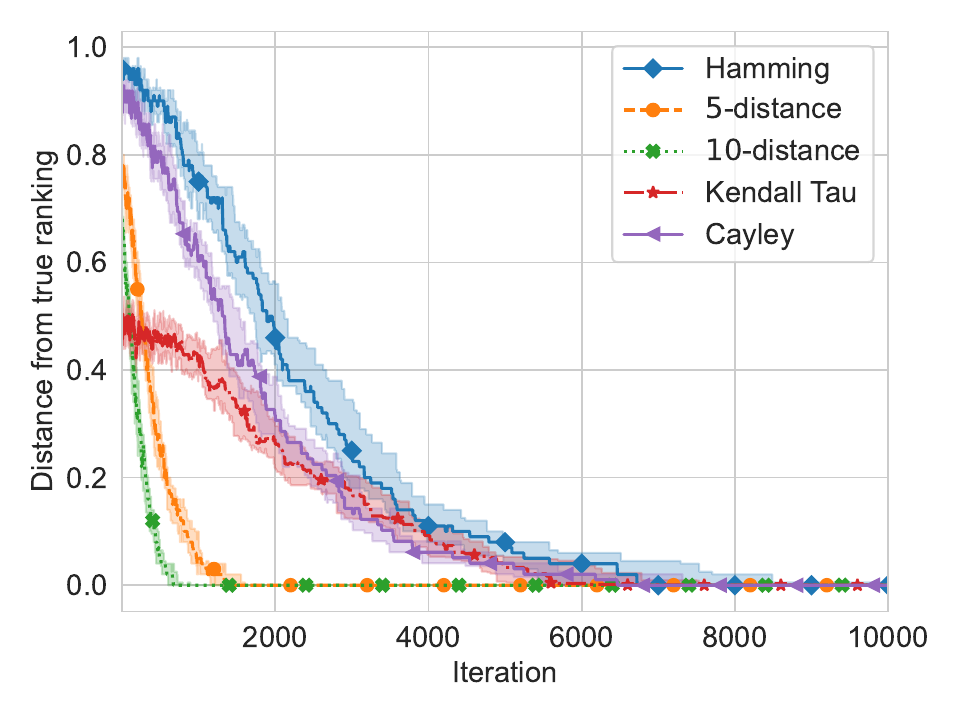}
    \caption{KaczRank run for 10000 iterations on a set of full observations corresponding to a ranking of $n = 50$ objects. We plot the normalized distances at each iteration for the Hamming, $5-$, $10-$, Kendall tau, and Cayley distances.}
    \label{fig:all_dists}
\end{figure}

We note that there are other distances used within the literature: for example, the Kendall tau distance (also known as the bubble-sort distance) computes the number of neighbourly transpositions required to permute one ranking into another, and the Cayley distance computes the number of transpositions required to permute one ranking into another. In \cref{fig:all_dists}, we show an example of the convergence of KaczRank under these distances, alongside the Hamming, $5-$, and $10-$distances for a collection of $50$ objects. To aid the visual comparison, we normalize each distance to take values between $0$ and $1$. We see that KaczRank converges under all distances, and for the remainder of this section, we restrict our attention to the Hamming and $k-$distances.

In all experiments, we mark the median across trials with the interquartile range shaded. Experiments were run for $20$ trials with $\varepsilon = 10^{-5}$ unless otherwise noted. Our code is available at \href{https://github.com/alexandersietsema/KaczRank}{https://github.com/alexandersietsema/KaczRank}.

\subsubsection{Consistent Data}
\label{sec:consistent}
First, we consider the case where comparisons are sampled from the full set of possible pairings, and every sampled comparison respects the underlying ranking (i.e., there is no noise). \cref{fig:vanilla} showcases the behaviour of the KaczRank method on such a system with $50$ objects, in terms of the Hamming and $k$-distance. We observe convergence to the true ranking under all notions of distance, where the rate of convergence is greater for higher values of $k$. This is to be expected as the $k$-distance is a relaxation of the Hamming distance, with the relaxation being greater for higher values of $k$.

\begin{figure}
    \centering
    \includegraphics[width=2in]{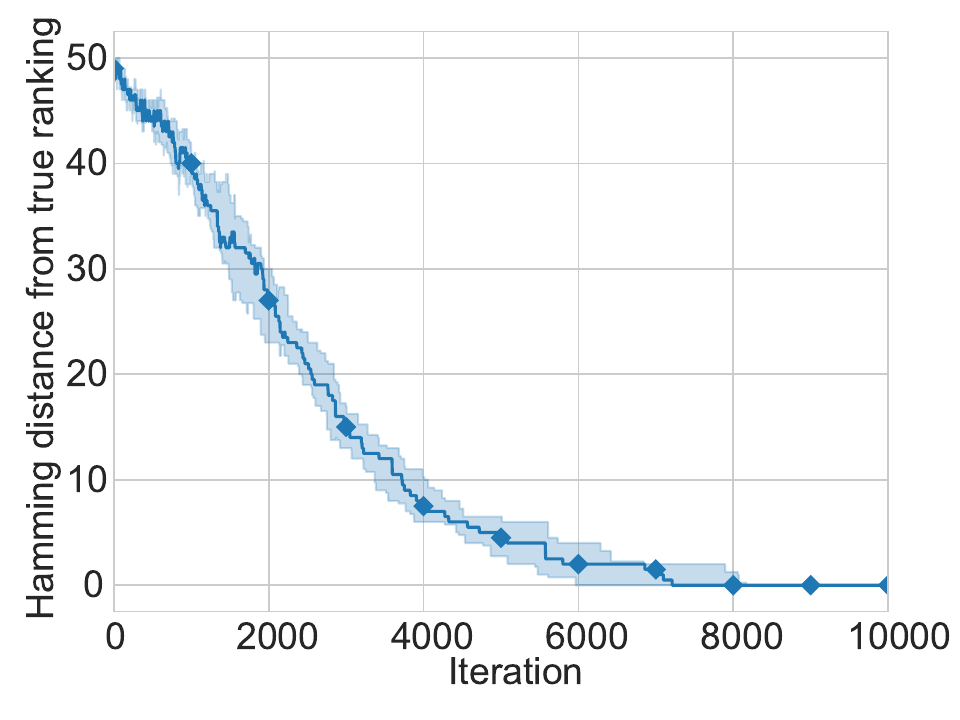} \quad \includegraphics[width=2in]{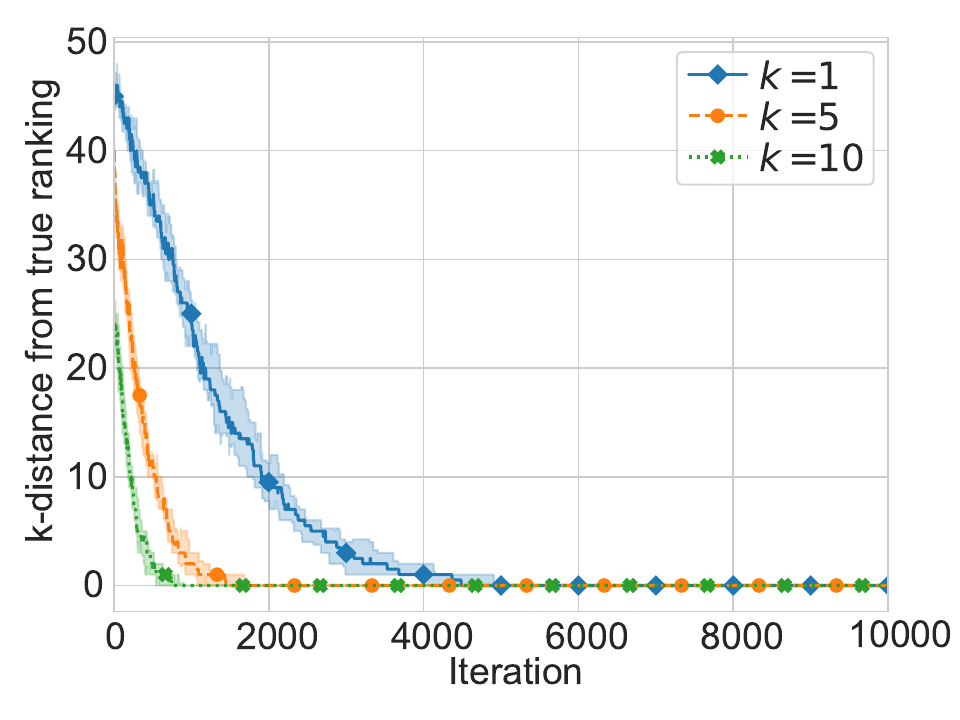}
    \caption{KaczRank run on a set of full observations corresponding to a ranking of $n=50$ objects for 10000 iterations. Left: Hamming distance versus iteration. Right: $k$-distance versus iteration for $k=1,5,10$.}
    \label{fig:vanilla}
\end{figure}

\cref{fig:vanillapartial} demonstrates the performance of KaczRank in the case where comparisons are sampled from a subset of the full set of possible pairings. That is, for $q \in (0,1]$, $\left\lfloor q{50 \choose 2}\right\rfloor$ comparisons are chosen uniformly from the full set of ${50 \choose 2}$ comparisons, and at each iteration of our methods, a comparison is sampled from this subset. We compare the Hamming and $k$-distances between our iterate and the true ranking after $10000$ iterations, for $q$ ranging from $0.05$ to $1$. We see that the true ranking is unlikely to be recoverable unless $q=1$, but that the $k$-distance is more robust to this form of incomplete data. For example, we see that only around half of the data is needed to obtain a ranking where each object is within $5$ spots of its true rank. 

\begin{figure}
    \centering
    \includegraphics[width=2in]{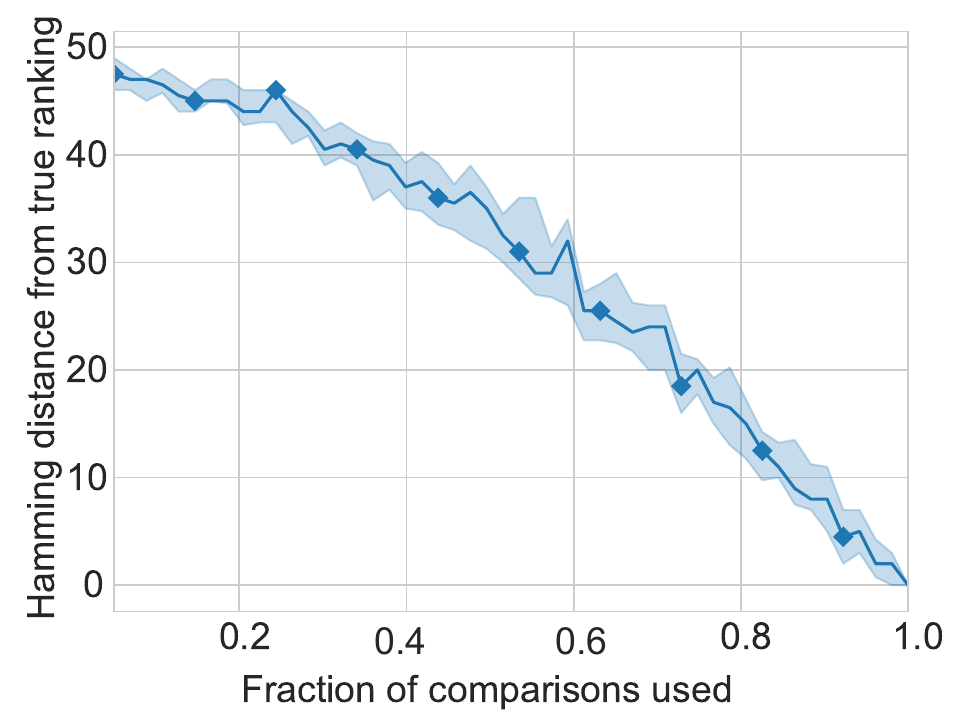} \quad \includegraphics[width=2in]{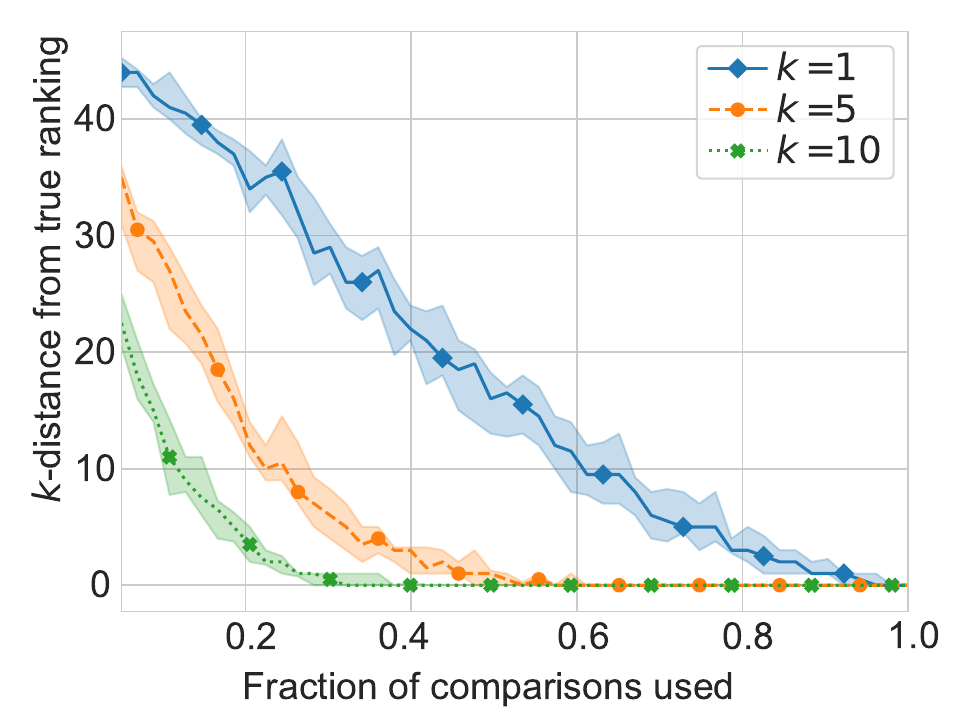}
    \caption{KaczRank run on a subset of the full set of pairwise comparisons corresponding to a ranking of $n=50$ objects, constructed by sampling a fraction $q \in (0,1]$ of the full set uniformly at random without replacement. Left: Hamming distance after $10000$ iterations versus $q$. Right: $k$-distance after $10000$ iterations versus $q$ with $k=1,5,10$.}
    \label{fig:vanillapartial}
\end{figure}

Our main theoretical result, \cref{thm:main_no_noise}, gives a bound on the expected number of iterations for KaczRank applied to a full set of comparisons of $n$ objects to converge, of the form $n^{\mathcal{O}(n^2)}$. In practice, however, the expected number of iterations grows far slower. In \cref{fig:n_iters}, we plot the number of iterations required to obtain the true ranking versus $n\in \{5,10,20,50,100,200\}$, and with the aid of the log-log scale we see that growth is closer to $n^{\mathcal{O}(1)}$.

\begin{figure}
    \centering
    \includegraphics[width=3in]{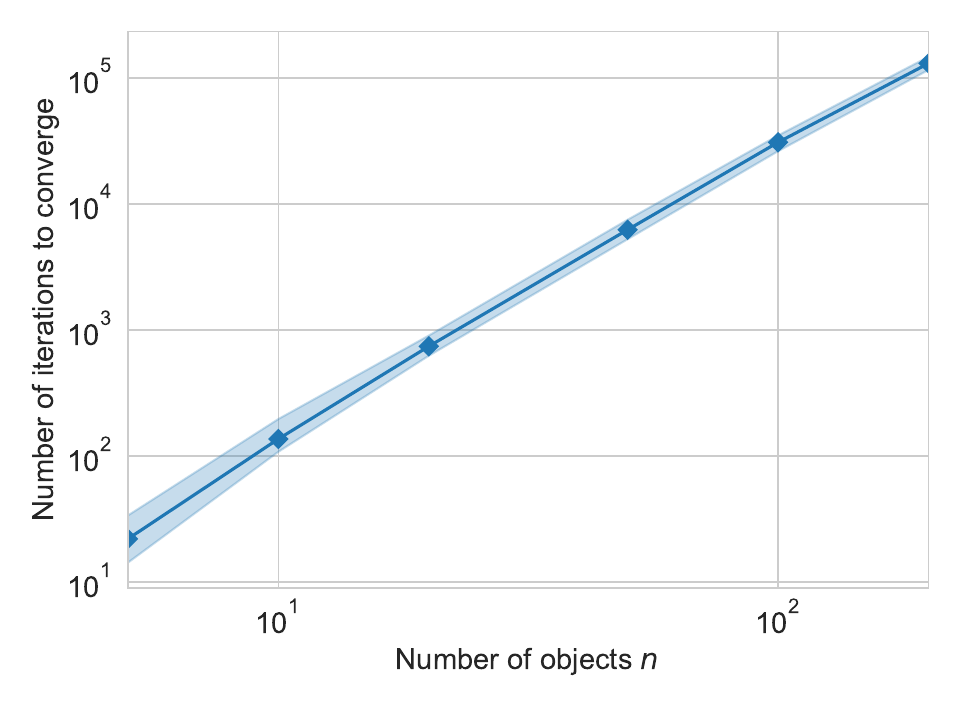}
    \caption{KaczRank run on a set of full observations corresponding to a ranking of $n$ objects for $n \in [5, 200]$ across $50$ trials. We plot the number of iterations of KaczRank required to converge to the true ranking versus the number of objects $n$, on a log-log scale.}
    \label{fig:n_iters}
\end{figure}

\subsubsection{Inconsistent Data}
\label{sec:inconsistent}
In the next set of experiments, we consider inconsistent observations as described in \cref{sec:noisy}, where at each iteration the sampled pairwise comparison is reversed with probability $p \in [0,1/2)$ (that is, if the true ranking has object $i$ ranked higher than object $j$, with probability $p$ we will actually sample the incorrect comparison $j > i$). We assume these flips occur independently across iterations. We begin by offering some insights into the choice of cautiousness parameter for our method designed for this setting, CautiousRank, as detailed in \cref{alg:cautiousrank}. In \cref{fig:alphatuning} we show the Hamming distance between the iterate produced by CautiousRank and the true ranking after $10000$ iterations, for a range of cautiousness parameters $\alpha$ and flip probabilities $p$. It is apparent that if $\alpha$ is set to be too small, the method can make no progress towards the true ranking. On the other hand, if $\alpha$ is too large the method approaches KaczRank and also fails to come close to the true ranking. However, there is a range of $\alpha$ that enables CautiousRank to outperform KaczRank, and our plots suggest that the optimal $\alpha$ does not depend on the flip probability $p$.

\begin{figure}
    \centering
    \includegraphics[width=2in]{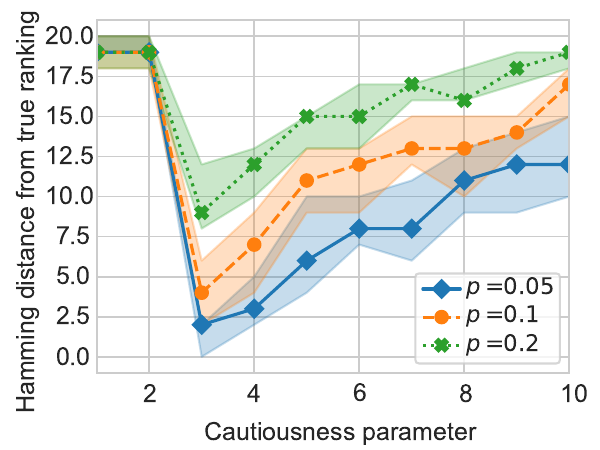} \quad \includegraphics[width=2in]{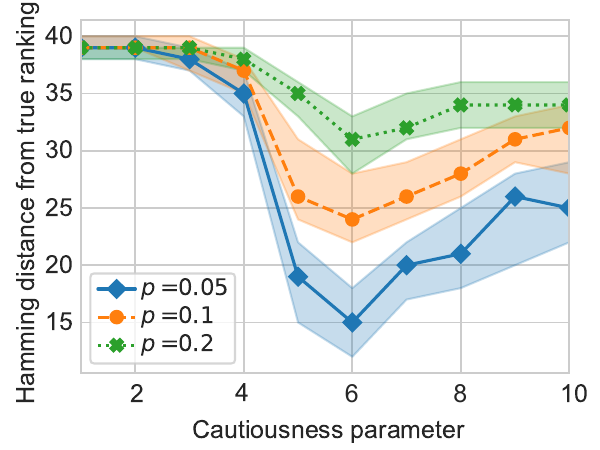}
    \caption{CautiousRank run on the full set of pairwise comparisons corresponding to a ranking of $n=20$ (left) and $n=40$ (right) objects, for a range of cautiousness parameters $\alpha$ across $25$ trials. We plot the Hamming distance from the true ranking after $10000$ iterations versus $\alpha$, for a range of flip probabilities $p$.}
    \label{fig:alphatuning}
\end{figure}

For $n=20$ objects, we first consider the full information case, in which any of the ${20 \choose 2}$ comparisons may be sampled at any iteration.In \cref{fig:noisyhamming} we show the effect of varying $p$ between $0$ and $0.3$ on the Hamming distance between the KaczRank/CautiousRank iterate and the true ranking after $10000$ iterations. We see that as soon as any flipped observations are introduced, the performance of KaczRank breaks down and the true ranking is not recoverable. However, CautiousRank (with $\alpha = 4$) is more robust to these noisy samples, and outputs iterates that are relatively close to the true ranking. This is confirmed when looking at the $k$-distance plots in \cref{fig:noisykdist}: we see that CautiousRank is able to return a ranking where every element is within $5$ spots of its true position even for very noisy data. 

Lastly, we apply CautiousRank in a setting in which we have access only to a subset of the full set of comparisons, and also in which each sampled comparison has some probability of being flipped. \cref{fig:FFA} shows the results of applying CautiousRank with $\alpha = 4$ to a system formed from $n = 20$ objects, with flipping probabilities $p = 0.05$ (bottom row) and $0.1$ (top row), where we vary the proportion of available comparisons. We see that, as in the consistent setting, one still requires access to a large majority of the total comparisons in order to obtain a ranking close to the true ranking under the Hamming distance. 
We observe also that (as is to be expected), as the flipping probability $p$ increases, the quality of the ranking produced by CautiousRank decreases across all metrics.

\begin{figure}
    \centering
    \includegraphics[width=3in]{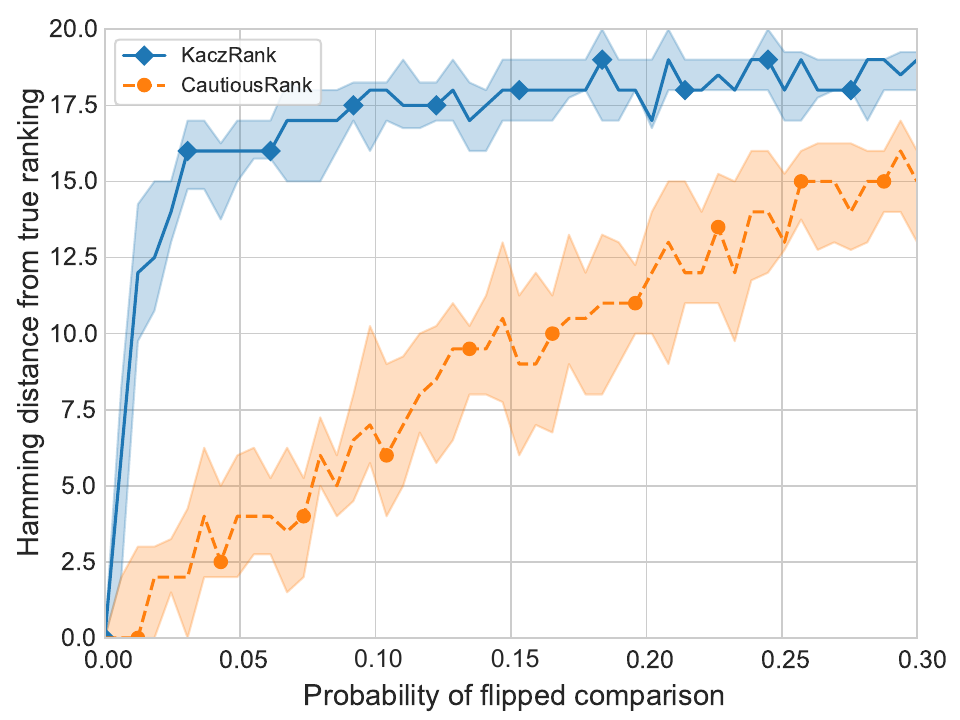}
    \caption{Both approaches run for $10000$ iterations versus $p$ on a set of full observations corresponding to a ranking of $n=20$ objects, where each sampled comparison has some probability $p$ of being flipped.}
    \label{fig:noisyhamming}
\end{figure}

\begin{figure}
    \centering
    \includegraphics[width=2in]{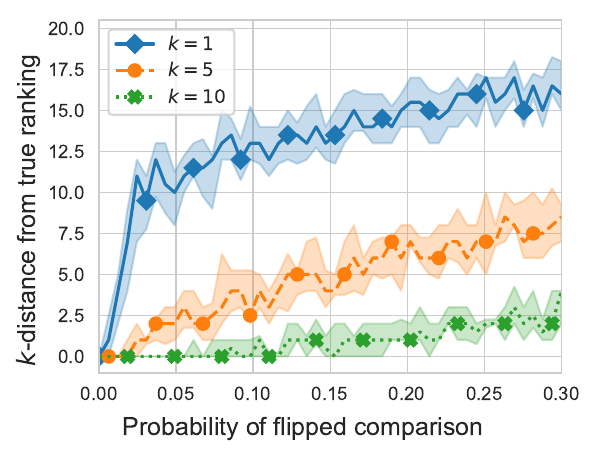} \quad \includegraphics[width=2in]{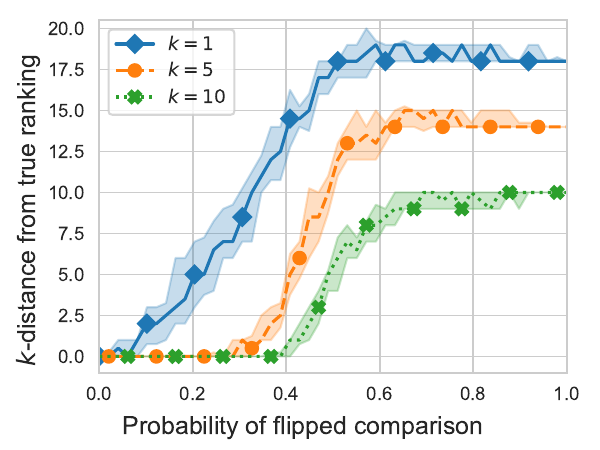}
    \caption{Both approaches run on a set of full observations corresponding to a ranking of $n=20$ objects, where each sampled comparison has some probability $p$ of being flipped. Left: $k$-distance after $10000$ iterations of KaczRank, versus $p$, with $k=1,5,10$. Right: $k$-distance after $10000$ iterations of CautiousRank, versus $p$, with $k=1,5,10$. Note the wider range of $p$ on the right, to show the extent of the robustness of CautiousRank.}
    \label{fig:noisykdist}
\end{figure}

\begin{figure}
    \centering
    \includegraphics[width=2in]{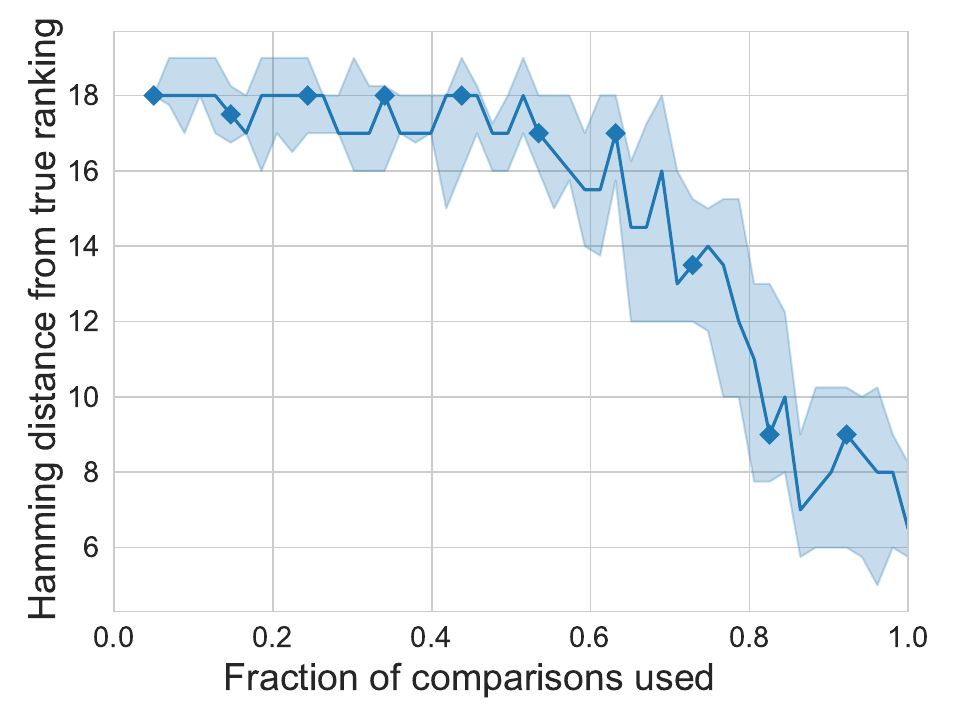} \quad \includegraphics[width=2in]{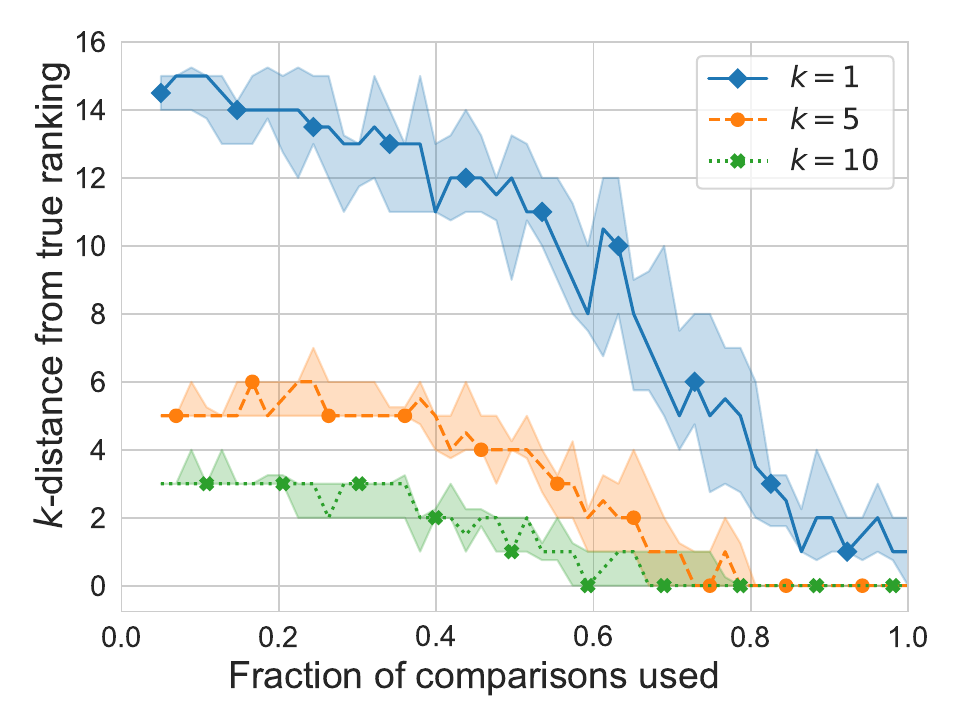} \\ \includegraphics[width=2in]{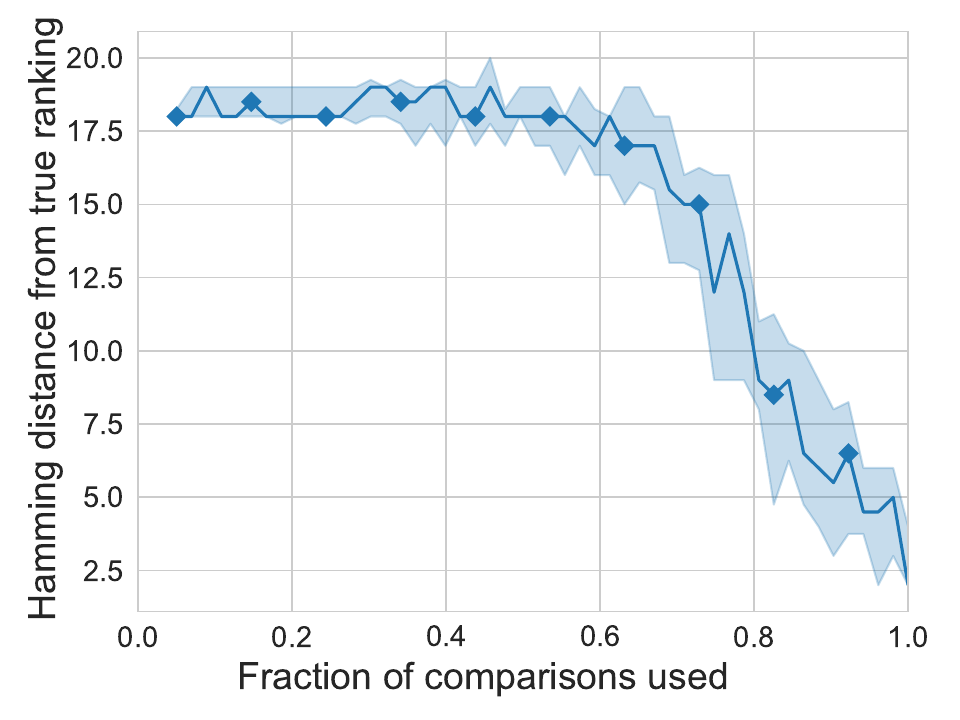} \quad 
    \includegraphics[width=2in]{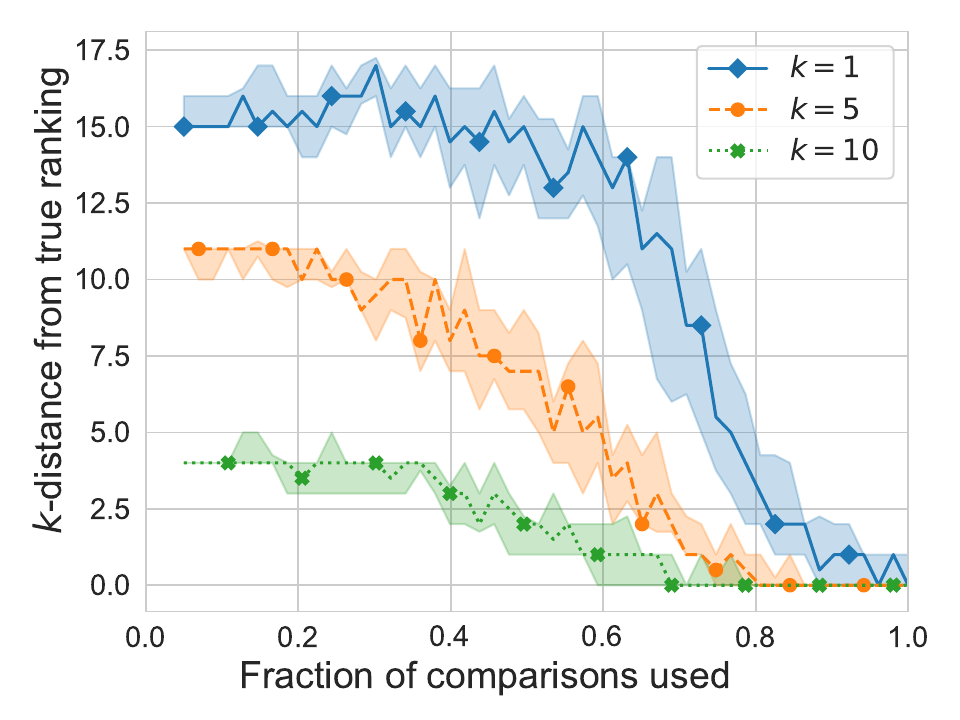}
    \caption{CautiousRank run on a subset of the full set of pairwise comparisons corresponding to a ranking of $n=20$ objects, constructed by a sampling a fraction $q \in (0,1]$ of the full set uniformly at random without replacement, where each comparison has probability $p = 0.1$ (top row) or $p = 0.05$ (bottom row) of being flipped. The Hamming (left column) and $k$-distance (right column) between the CautiousRank iterate and the true ranking after $10000$ iterations are plotted versus $q$.}
    \label{fig:FFA}
\end{figure}

\subsection{Comparisons To Other Methods}
\label{sec:comparisons}
In this section, we compare the computational time and memory usage of KaczRank to other pairwise comparison algorithms with full consistent data. In particular, we compare to the Luce Spectral Ranking (LSR) and iterative Luce Spectral Ranking algorithms \cite{maystre2015fast} as well as the similar Rank Centrality algorithm \cite{negahban2012iterative}. These methods are both popular and have readily available implementations provided by the \texttt{choix} Python package \cite{choix}. Experiments were run on a computer with an Intel i7-7700HQ processor running at 2.80 GHz using 16 GB of RAM. Memory usage was recorded using \texttt{tracemalloc} standard library package. Note that both the time and memory usage values are approximate, though the observed trends are clear and match theoretical expectations.

\begin{figure}
    \centering
    \includegraphics[width=2in]{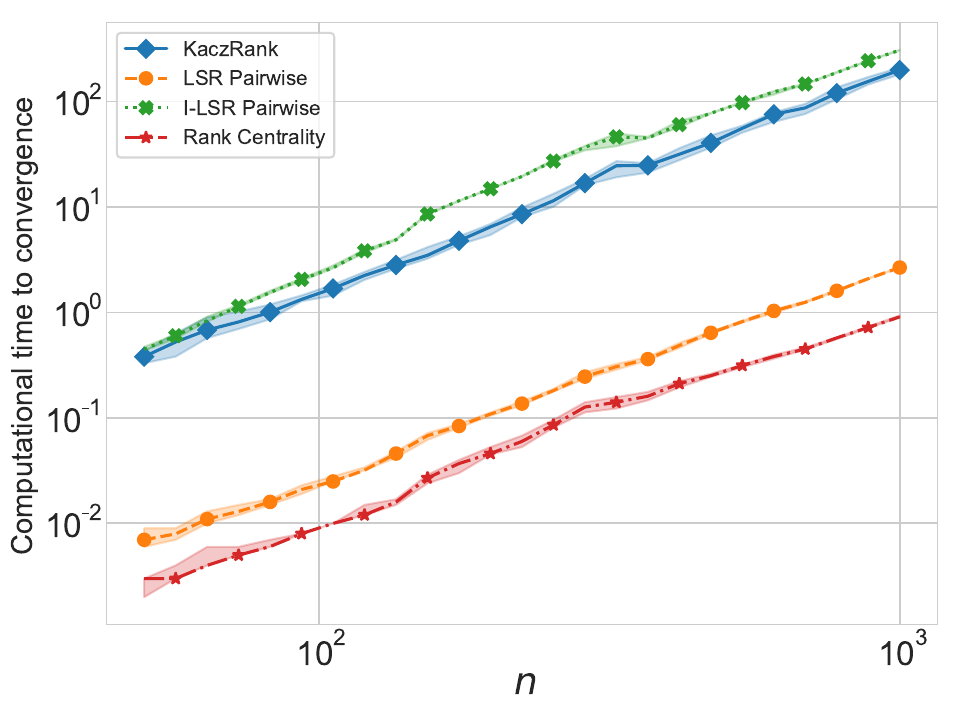} \quad \includegraphics[width=2in]{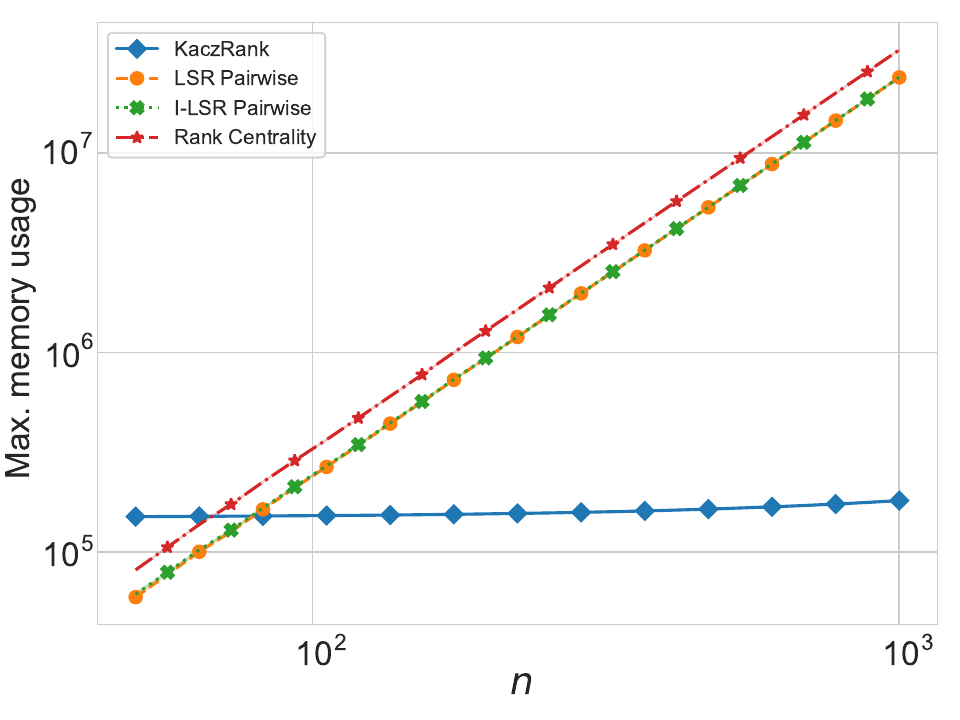}\textbf{}
    \caption{KaczRank compared to LSR, I-LSR, and Rank Centrality on a set of full observations across 25 trials on a log-log scale. Left: average computational time for convergence versus the number of objects $n$. Left: average memory usage versus the number of objects $n$.}
    \label{fig:comparison}
\end{figure}

In \cref{fig:comparison} we compare the average computational time and average approximate maximum memory usage across 25 trials for each algorithm as the number of objects $n$ increases. KaczRank was iterated until convergence to the true ranking. We note that because the comparison algorithms require that each object have a transitive win over every other object, we provide a small amount of regularization to each as provided by the \texttt{choix} package to allow for convergence to the true ranking. We see that KaczRank is roughly two orders of magnitude slower in computational time compared to LSR and Rank Centrality, but has significantly smaller local memory costs than any of the other algorithms. As the comparison algorithms require the storage of an $n \times n$ weight matrix, we expect those methods to scale quadratically in memory while KaczRank scales only linearly.  This is, of course, not surprising, since Kaczmarz methods are used precisely because of their low memory costs.

We remark that we do not provide comparisons of CautiousRank to methods on inconsistent data. Because CautiousRank, unlike the methods we are comparing against, is not designed to converge with data generated with the BTL model, it is difficult to design an experiment with inconsistent data without implicitly favoring one of the two generative frameworks. We do expect CautiousRank to provide the same improved local memory requirements as KaczRank, though its computational time is significantly slower because the iterate ranking must be re-computed at each step. This problem may be solved by computing only an approximate comparison between the rankings of successive iterates in line 5 of \cref{alg:cautiousrank}; this is beyond the scope of our analysis but may be interesting future work. Our goal with these comparisons is not to demonstrate that KaczRank and related methods are immediate improvements on existing algorithms, but rather to demonstrate the efficacy of Kaczmarz-type methods in limited-memory settings.

\section{Conclusion}
\label{sec:conclusion}
We have analyzed several variants of stochastic gradient descent methods applied to data stemming from pairwise comparisons of a finite set of objects. Assuming some true underlying ranking, we identify mathematically and empirically when such methods converge to a feasible point that reveals the underlying ranking, or an approximation to the ranking.  We believe this is a first step toward further understanding when such iterative methods can be applied to discrete mathematical problems.

\bibliographystyle{spmpsci}      
\bibliography{main}   

\begin{thebibliography}{10}
\providecommand{\url}[1]{{#1}}
\providecommand{\urlprefix}{URL }
\expandafter\ifx\csname urlstyle\endcsname\relax
  \providecommand{\doi}[1]{DOI~\discretionary{}{}{}#1}\else
  \providecommand{\doi}{DOI~\discretionary{}{}{}\begingroup
  \urlstyle{rm}\Url}\fi

\bibitem{agarwal2020rank}
Agarwal, A., Agarwal, S., Khanna, S., Patil, P.: Rank aggregation from pairwise
  comparisons in the presence of adversarial corruptions.
\newblock In: Int. Conf. Mach. Learn., pp. 85--95. PMLR (2020)

\bibitem{AggRecommender}
Aggarwal, C.C.: Recommender systems: the textbook (2013)

\bibitem{agmon1954relaxation}
Agmon, S.: The relaxation method for linear inequalities.
\newblock Can. J. Math. \textbf{6}, 382--392 (1954)

\bibitem{Borkar2013Rand}
Borkar, V.S., Karamchandani, N., Mirani, S.: Randomized {K}aczmarz for rank
  aggregation from pairwise comparisons.
\newblock In: IEEE Proc. Inf. Theory Workshop, pp. 389--393 (2016).
\newblock \doi{10.1109/ITW.2016.7606862}

\bibitem{cai2012exponential}
Cai, Y., Zhao, Y., Tang, Y.: Exponential convergence of a randomized {K}aczmarz
  algorithm with relaxation.
\newblock In: Proc. Int. Congr. on Comp. Appl. Comp. Sci., pp. 467--473.
  Springer (2012)

\bibitem{Bullo2020Conv}
Chen, G., Su, W., Mei, W., Bullo, F.: Convergence properties of the
  heterogeneous deffuant–weisbuch model.
\newblock Automatica \textbf{114}, 108,825 (2020).
\newblock \doi{https://doi.org/10.1016/j.automatica.2020.108825}.
\newblock
  \urlprefix\url{https://www.sciencedirect.com/science/article/pii/S0005109820300236}

\bibitem{godsil2001algebraic2}
Godsil, C., Royle, G.F.: Algebraic graph theory (2001)

\bibitem{Gower2021OnAda}
Gower, R.M., Molitor, D., Moorman, J., Needell, D.: On adaptive
  sketch-and-project for solving linear systems.
\newblock SIAM J. Matrix Anal. Appl. \textbf{42}(2), 954--989 (2021)

\bibitem{Gower2015Random}
Gower, R.M., Richtárik, P.: Randomized iterative methods for linear systems.
\newblock SIAM J. Matrix Anal. Appl. \textbf{36}(4), 1660--1690 (2015)

\bibitem{Haddock2021Greed}
Haddock, J., Ma, A.: Greed works: An improved analysis of sampling
  {Kaczmarz-Motzkin}.
\newblock SIAM J. Math. Data Sci. \textbf{3}(1), 342--368 (2021)

\bibitem{quantHNRS20}
Haddock, J., Needell, D., Rebrova, E., Swartworth, W.: Quantile-based iterative
  methods for corrupted systems of linear equations.
\newblock SIAM J. Matrix Anal. Appl. \textbf{43}(3), 605--637 (2022)

\bibitem{HeckelApprox}
Heckel, R., Simchowitz, M., Ramchandran, K., Wainwright, M.: Approximate
  ranking from pairwise comparisons (2018).
\newblock \urlprefix\url{https://proceedings.mlr.press/v84/heckel18a.html}

\bibitem{herbrich2000large}
Herbrich, R.: Large margin rank boundaries for ordinal regression.
\newblock Adv. Large Margin Classifiers pp. 115--132 (2000)

\bibitem{herbrich1999support}
Herbrich, R., Graepel, T., Obermayer, K.: Support vector learning for ordinal
  regression.
\newblock Proc. Int. Conf. Artif. Int.  (1999)

\bibitem{Herman1993Algebr}
Herman, G., Meyer, L.: Algebraic reconstruction techniques can be made
  computationally efficient (positron emission tomography application).
\newblock IEEE Trans. Med. Imaging \textbf{12}(3), 600--609 (1993)

\bibitem{hodgkinson2021multiplicative}
Hodgkinson, L., Mahoney, M.: Multiplicative noise and heavy tails in stochastic
  optimization.
\newblock In: Int. Conf. Mach. Learn., pp. 4262--4274. PMLR (2021)

\bibitem{DeLoeraMotzkin16}
J.~A. De~Loera J.~Haddock, D.N.: A sampling {K}aczmarz-{M}otzkin algorithm for
  linear feasibility.
\newblock SIAM J. Sci. Comput. \textbf{39}(5) (2017)

\bibitem{joachims2002optimizing}
Joachims, T.: Optimizing search engines using clickthrough data.
\newblock In: Proc. SIGKDD Int. Conf. Knowl. Discov. Data Min., pp. 133--142
  (2002)

\bibitem{Kaczmarz1937Angen}
Kaczmarz, S.: Angen\"aherte aufl\"osung von systemen linearer gleichungen.
\newblock Bull. Int. Acad. Pol. Sci. Let. \textbf{35}, 355--357 (1937)

\bibitem{Leventhal2010RandomizedMF}
Leventhal, D., Lewis, A.S.: Randomized methods for linear constraints:
  Convergence rates and conditioning.
\newblock Math. Oper. Res. \textbf{35}, 641--654 (2010)

\bibitem{choix}
Maystre, L.: choix.
\newblock \url{https://pypi.org/project/choix/} (2022)

\bibitem{maystre2015fast}
Maystre, L., Grossglauser, M.: Fast and accurate inference of
  {P}lackett--{L}uce models.
\newblock Adv. Neural Inf. Process. Syst. \textbf{28} (2015)

\bibitem{Necoara2019Faster}
Necoara, I.: Faster randomized block {Kaczmarz} algorithms.
\newblock SIAM J. Matrix Anal. Appl. \textbf{40}(4), 1425--1452 (2019)

\bibitem{needell2010randomized}
Needell, D.: Randomized {K}aczmarz solver for noisy linear systems.
\newblock BIT Numer. Math. \textbf{50}(2), 395--403 (2010)

\bibitem{Needell2014Paved}
Needell, D., Tropp, J.: Paved with good intentions: Analysis of a randomized
  block {Kaczmarz} method.
\newblock Linear Algebra Appl. \textbf{441}, 199--221 (2014)

\bibitem{negahban2012iterative}
Negahban, S., Oh, S., Shah, D.: Iterative ranking from pair-wise comparisons.
\newblock Adv. Neural Inf. Process. Syst. \textbf{25} (2012)

\bibitem{PenaHoffman}
Pe\~{n}a, J., Vera, J.C., Zuluaga, L.F.: New characterizations of hoffman
  constants for systems of linear constraints.
\newblock Math. Program. \textbf{187}(1–2), 79–109 (2021).
\newblock \doi{10.1007/s10107-020-01473-6}.
\newblock \urlprefix\url{https://doi.org/10.1007/s10107-020-01473-6}

\bibitem{Piech2013Tuned}
Piech, C., Huang, J., Chen, Z., Do, C., Ng, A., Koller, D.: Tuned models of
  peer assessment in moocs.
\newblock In: Int. Conf. Educ. Data Min. (2013)

\bibitem{radinsky2011ranking}
Radinsky, K., Ailon, N.: Ranking from pairs and triplets: Information quality,
  evaluation methods and query complexity.
\newblock In: Proc. Int. Conf. Web Search Data. Min., pp. 105--114 (2011)

\bibitem{Salesses2013TheCo}
Salesses, P., Schechtner, K., Hidalgo, C.A.: The collaborative image of the
  city: Mapping the inequality of urban perception.
\newblock PLoS One \textbf{8}(7), e68,400 (2013)

\bibitem{schopfer2022extended}
Sch{\"o}pfer, F., Lorenz, D.A., Tondji, L., Winkler, M.: Extended randomized
  {K}aczmarz method for sparse least squares and impulsive noise problems.
\newblock Linear Algebra Appl. \textbf{652}, 132--154 (2022)

\bibitem{Strohmer2009Arand}
Strohmer, T., Vershynin, R.: A randomized {Kaczmarz} algorithm with exponential
  convergence.
\newblock J. Fourier Anal. Appl. \textbf{15}(2), 262--278 (2009)

\bibitem{wauthier2013efficient}
Wauthier, F., Jordan, M., Jojic, N.: Efficient ranking from pairwise
  comparisons.
\newblock In: Int. Conf. Mach. Learn., pp. 109--117. PMLR (2013)

\bibitem{wu2019multiplicative}
Wu, J., Hu, W., Xiong, H., Huan, J., Zhu, Z.: The multiplicative noise in
  stochastic gradient descent: Data-dependent regularization, continuous and
  discrete approximation.
\newblock arXiv preprint arXiv:1906.07405  (2019)

\bibitem{zouzias2013randomized}
Zouzias, A., Freris, N.M.: Randomized extended {K}aczmarz for solving least
  squares.
\newblock SIAM J. Matrix Anal. Appl. \textbf{34}(2), 773--793 (2013)

\end{thebibliography}

\end{document}